\documentclass[12pt,twoside]{amsart}
\usepackage{amsmath, amsthm, amscd, amsfonts, amssymb, graphicx}
\usepackage[bookmarksnumbered, plainpages]{hyperref}
\usepackage[latin1]{inputenc}
\newtheorem{thm}{Theorem}[section]

\newtheorem{lem}[thm]{Lemma}
\newtheorem{prop}[thm]{Proposition}
\newtheorem{defn}[thm]{Definition}

\numberwithin{equation}{section}

\begin{document}

\title{\bf  $SL(2,{\bf Z})$ modular forms and anomaly cancellation formulas for almost complex manifolds}
\author{Yong Wang}

\thanks{{\scriptsize
\hskip -0.4 true cm \textit{2010 Mathematics Subject Classification:}
58C20; 57R20; 53C80.
 }}

\maketitle

\begin{abstract}
 In this paper, we define a generalized elliptic genus of an almost complex manifold with an extra complex bundle which generalize the elliptic genus in \cite{LiP}. This generalized elliptic genus is a generalized Jacobi form. By this generalized Jacobi form, we can get some $SL(2,{\bf Z})$ modular forms. By these $SL(2,{\bf Z})$ modular forms, we get some interesting
anomaly cancellation formulas for an almost complex manifold . As corollaries, we get some divisibility results of the holomorphic Euler
characteristic number.
\end{abstract}

\vskip 0.2 true cm


\pagestyle{myheadings}
\markboth{\rightline {\scriptsize Yong Wang}}
         {\leftline{\scriptsize $SL(2,{\bf Z})$ modular forms and anomaly cancellation formulas for almost complex manifolds}}

\bigskip
\bigskip


\section{ Introduction}
\quad For an arbitrary compact spin manifold one can define its elliptic genus. It is a modular form in one variable with respect to a congruence
subgroup of level $2$. For a compact complex manifold one can define its elliptic genus as a function in two complex variables. In the last case,
the elliptic genus is the holomorphic Euler characteristic of a formal power series with vector bundle coefficients. If the first Chern class of the complex manifold is equal to zero, then the elliptic genus is a weak Jacobi form. In \cite{LiP}, Li extended the elliptic genus of an almost complex
manifold to a twisted version where an extra complex vector bundle is involved. Under some conditions, Li proved this elliptic genus is a weak Jacobi
form. In this paper, we extend the Li's elliptic genus and prove this generalized elliptic genus is not a weak Jacobi form and we call it the generalized Jacobi form. By this generalized Jacobi form, we can get some $SL(2,{\bf Z})$ modular forms as in \cite{LiP}.\\
\indent In 1983, the physicists Alvarez-Gaum\'{e} and Witten \cite{AW}
  discovered the "miraculous cancellation" formula for gravitational
  anomaly which reveals a beautiful relation between the top
  components of the Hirzebruch $\widehat{L}$-form and
  $\widehat{A}$-form of a $12$-dimensional smooth Riemannian
  manifold. Kefeng Liu \cite{Li1} established higher dimensional "miraculous cancellation"
  formulas for $(8k+4)$-dimensional Riemannian manifolds by
  developing modular invariance properties of characteristic forms.
  These formulas could be used to deduce some divisibility results. In
  \cite{HZ1}, \cite{HZ2}, \cite{CH}, some more general cancellation formulas that involve a
  complex line bundle and their applications were established. In \cite{HLZ2}, Han, Liu and Zhang showed that both of the Green-Schwarz anomaly factorization formula
for the gauge group $E_8\times E_8$ and the Horava-Witten anomaly factorization formula for the gauge
group $E_8$ could be derived through modular forms of weight $14$. This answered a question of J.
H. Schwarz.  In \cite{HHLZ}, Han, Huang, Liu and Zhang introduced a modular form of weight $14$ over $SL(2,{\bf Z})$ and a modular form of weight $10$ over $SL(2,{\bf Z})$ and they got some interesting
anomaly cancellation formulas on $12$-dimensional manifolds. In \cite{Wa}, by some $SL(2,{\bf Z})$ modular forms introduced in \cite{Li2} and \cite{CHZ} , we get some interesting
anomaly cancellation formulas. As corollaries, we get some divisibility results of index of twisted Dirac operators.
 {\bf Our motivation is to prove more anomaly cancellation formulas for almost complex manifolds by modular forms over $SL(2,{\bf Z})$ induced by the above generalized Jacobi form. }\\
 \indent This paper is organized as follows: In Section 2, we introduce the generalized elliptic genus and prove it is a generalized Jacobi form.
 In Section 3, by this generalized Jacobi form, we can get some $SL(2,{\bf Z})$ modular forms as in \cite{LiP}.  By these $SL(2,{\bf Z})$ modular forms, we get some interesting anomaly cancellation formulas for an almost complex manifold . As corollaries, we get some divisibility results of the holomorphic Euler
characteristic number.

 \vskip 1 true cm

\section{Generalized elliptic genus for almost complex manifolds}

\indent Let $(M,J)$ be a $2d$-dimensional almost manifold and $T$ be the holomorphic tangent bundle in the sense of $J$ and $T^*$ is the dual of $T$. Let $W$ denote a complex $l$-dimensional vector bundle on $M$. Denote the first Chern classes of $T$ and $W$ by $c_1(M)$ and $c_1(W)$. We denote by
$2\pi \sqrt{-1}x_i~(1\leq i\leq d)$ and $2\pi \sqrt{-1}w_j~(1\leq j\leq l)$ respectively the formal Chern roots of $T$ and $W$. Then the Todd form of $(M,J)$ is defined by
\begin{equation}
  {\rm Td}(M):=\prod_{i=1}^d\frac{2\pi \sqrt{-1}x_i}{1-e^{-2\pi \sqrt{-1}x_i}}.
\end{equation}
Let $(\tau,z)\in \mathcal{H}\times \mathcal{C}$ where $\mathcal{H}$ is the upper half plane and $\mathcal{C}$ is the complex plane. Let $a_0\geq 1$ be a positive integer. Let $y_r=e^{2\pi \sqrt{-1}m_rz}$ for $1\leq r\leq a_0$ and the positive integer $m_r$. Let $q= e^{2\pi \sqrt{-1}\tau}$
and $c:=\prod_{j=1}^{\infty}(1-q^j)$.\\
\indent For any complex number $t$, let
  \begin{equation} \wedge_t(E)={\bf C}|_M+tE+t^2\wedge^2(E)+\cdots,~S_t(E)={\bf
   C}|_M+tE+t^2S^2(E)+\cdots\end{equation}
   denote respectively the total exterior and symmetric powers of
   $E$, which live in $K(M)[[t]].$ The following relations between
   these operations hold,
   $$S_t(E)=\frac{1}{\wedge_{-t}(E)},~\wedge_t(E-F)=\frac{\wedge_t(E)}{\wedge_t(F)}.\eqno(2.3)$$
   Moreover, if $\{\omega_i\},\{\omega_j'\}$ are formal Chern roots
   for Hermitian vector bundles $E,F$ respectively, then
 \begin{equation}
 {\rm ch}(\wedge_t(E))=\prod_i(1+e^{\omega_i}t),
  ~~{\rm ch}(S_t(E))=\frac{1}{\prod_i(1-e^{\omega_i}t)}.
\end{equation}
\begin{defn}
The generalized elliptic genus of $(M^{2d},J)$ with respect to $W$, which we denote by ${\rm Ell}(M,W,\tau,z)$ is defined by
\begin{equation}
{\rm Ell}(M,W,\tau,z):=\left\{exp(\frac{a_0c_1(W)-c_1(M)}{2})Td(M){\rm ch}(E(M,W,\tau,z))\right\}^{(2d)}£¬
\end{equation}
where
\begin{align}
{\rm E}(M,W,\tau,z):=&c^{2(d-la_0)}y_1^{-\frac{l}{2}}\cdots y_{a_0}^{-\frac{l}{2}} \bigotimes _{n=1}^{\infty}
\left( \bigotimes _{r=1}^{a_0}\wedge_{-y_rq^{n-1}}(W^*)\wedge_{-y_r^{-1}q^{n}}(W)\right)\\\notag
&\bigotimes\left(\bigotimes _{n=1}^{\infty}S_{q^n}(T^*)\otimes S_{q^n}(T)\right).
\end{align}
\end{defn}
When $a_0=1$ and $m_1=1$, we get the Li's elliptic genus. We know that our elliptic genus is not the special case of the Li's elliptic genus since
$\wedge_{t_1}(W)\otimes \wedge_{t_2}(W)\neq \wedge_{t_1+t_2}(W)$ and $\wedge_{t_1}(W)\otimes \wedge_{t_2}(W)\neq \wedge_{t_1t_2}(W)$.
Using the same calculations as in Lemma 3.4 in \cite{LiP}, we have

\begin{lem}We have
\begin{align}
{\rm Ell}(M,W,\tau,z)=\left (\eta(\tau)^{3(d-la_0)}\prod_{i=1}^d\frac{2\pi \sqrt{-1}x_i}{\theta(\tau,x_i)}\prod_{j=1}^l\prod_{r=1}^{a_0}\theta
(\tau,w_j-m_rz)\right)^{(2d)},
\end{align}
where
\begin{equation}  \theta(\tau,z)=2q^{\frac{1}{8}}{\rm sin}(\pi
   z)\prod_{j=1}^{\infty}[(1-q^j)(1-e^{2\pi\sqrt{-1}z}q^j)(1-e^{-2\pi\sqrt{-1}z}q^j)],
   \end{equation}
 \begin{equation}
 \eta(\tau):=q^{\frac{1}{24}}\cdot c=q^{\frac{1}{24}}\prod_{j=1}^{\infty}(1-q^j).
   \end{equation}
\end{lem}

\begin{thm}
If $c_1(W)=0$ and the first Pontrjagin classes $p_1(M)=a_0p_1(W)$, then the generalized elliptic genus ${\rm Ell}(M,W,\tau,z)$
satisfies
\begin{align}
{\rm Ell}(M,W,\frac{a\tau+b}{c\tau+d_0},\frac{z}{c\tau+d_0})=(c\tau+d_0)^{d-la_0}{\rm exp}(\pi\sqrt{-1}l(\sum_{r=1}^{a_0}m^2_r)\frac{cz^2}{c\tau+d_0}){\rm Ell}(M,W,\tau,z);
\end{align}
\begin{align}
{\rm Ell}(M,W,\tau,z+\lambda \tau+\mu)=(-1)^{\mu l(\sum_{r=1}^{a_0}m_r)+\lambda la_0}{\rm exp}(-\pi\sqrt{-1}l(\sum_{r=1}^{a_0}m^2_r)(2\lambda z+\lambda^2\tau))
{\rm Ell}(M,W,\tau,z),
\end{align}
where $\left(\begin{array}{cc}
\ a & b  \\
 c  & d_0
\end{array}\right)\in SL(2,Z)$ and $\lambda,\mu \in Z$.
We know that the generalized elliptic genus ${\rm Ell}(M,W,\tau,z)$ is not a Jacobi form and we called it the generalized Jacobi form.
\end{thm}
\begin{proof}
By the following transformation laws:
\begin{align}
&\eta^3(-\frac{1}{\tau})=(\frac{\tau}{\sqrt{-1}})^{\frac{3}{2}}\eta^3(\tau),~~\eta^3(\tau+1)=e^{\frac{\pi \sqrt{-1}}{4}}\eta^3(\tau),
\\\notag
&\theta(\tau,z+1)=-\theta(\tau,z),~~\theta(\tau,z+\tau)=-q^{-\frac{1}{2}}{\rm exp}(-2\pi \sqrt{-1}z)\theta(\tau,z),
\\\notag
&\theta(\tau+1,z)=e^{\frac{\pi \sqrt{-1}}{4}}\theta(\tau,z),~~\theta(-\frac{1}{\tau},z)=-\sqrt{-1}(\frac{\tau}{\sqrt{-1}})^{\frac{1}{2}}
{\rm exp}(\pi \sqrt{-1}\tau z^2)\theta(\tau,\tau z),
\end{align}
we get ${\rm Ell}(M,W,\tau,z)$ satisfies the following transformation laws:
\begin{align}
&{\rm Ell}(M,W,\tau+1,z)={\rm Ell}(M,W,\tau,z),\\\notag
&{\rm Ell}(M,W,\tau,z+1)=(-1)^{l(\sum_{r=1}^{a_0}m_r)}{\rm Ell}(M,W,\tau,z),\\\notag
&{\rm Ell}(M,W,\tau,z+\tau)=(-1)^{la_0}{\rm exp}(-\pi\sqrt{-1}l(\sum_{r=1}^{a_0}m^2_r)(\tau+2z))
{\rm Ell}(M,W,\tau,z),\\\notag
&{\rm Ell}(M,W,-\frac{1}{\tau},\frac{z}{\tau})=\tau^{d-la_0}{\rm exp}(\pi\sqrt{-1}l(\sum_{r=1}^{a_0}m^2_r)\frac{z^2}{\tau})
{\rm Ell}(M,W,\tau,z).
\end{align}
By (2.12), we can (2.9) and (2.10).
\end{proof}
In the following, we introduce some generalized elliptic genus with extra complex bundle $W$ and real vector bundle $V$.
Let $V$ be a $2b_0$ dimensional real Euclidean vector bundle with the Euclidean connection $\nabla^V$ and the curvature $R^V$.  Let $\{\pm 2\pi\sqrt{-1}u_r\}~(1\leq r \leq b_0)$ be the formal Chern roots for $V\otimes C$. Let $\widetilde{V_C}=V_C-{\rm dim}V.$
\begin{defn}
The generalized elliptic genus of $(M^{2d},J)$ with respect to $W$ and $V$, which we denote by ${\rm Ell}(M,W,V,\tau,z)$, $\widetilde{{\rm Ell}}(M,W,V,\tau,z)$, $\overline{{\rm Ell}}(M,W,V,\tau,z)$ are defined by
\begin{align}
&{\rm Ell}(M,W,V,\tau,z):=\left\{exp(\frac{a_0c_1(W)-c_1(M)}{2})Td(M){\rm ch}(E(W,q,\tau))\right.\\\notag£¬
&\left.\cdot{\rm det}^{\frac{1}{2}}{\rm cosh}(\frac{\sqrt{-1}}{4 \pi}R^V){\rm ch}\left[
\bigotimes _{m=1}^{\infty}\wedge_{q^m}(\widetilde{V_C})
\otimes \bigotimes _{r=1}^{\infty}\wedge
_{q^{r-\frac{1}{2}}}(\widetilde{V_C})\otimes \bigotimes
_{s=1}^{\infty}\wedge _{-q^{s-\frac{1}{2}}}(\widetilde{V_C})\right]\right\}^{(2d)},
\end{align}
\begin{align}
&\widetilde{{\rm Ell}}(M,W,V,\tau,z):=\left\{exp(\frac{a_0c_1(W)-c_1(M)}{2})Td(M){\rm ch}(E(W,q,\tau))\right.\\\notag£¬
&\left.\cdot\left[{\rm det}^{\frac{1}{2}}{\rm cosh}(\frac{\sqrt{-1}}{4 \pi}R^V){\rm ch}\left[
\bigotimes _{m=1}^{\infty}\wedge_{q^m}(\widetilde{V_C})\right]
+{\rm ch}\left[\bigotimes _{r=1}^{\infty}\wedge
_{q^{r-\frac{1}{2}}}(\widetilde{V_C})\right]+ {\rm ch}\left[\bigotimes
_{s=1}^{\infty}\wedge _{-q^{s-\frac{1}{2}}}(\widetilde{V_C})\right]\right]\right\}^{(2d)},
\end{align}
\begin{align}
&\overline{{\rm Ell}}(M,W,V,\tau,z):=\left\{exp(\frac{a_0c_1(W)-c_1(M)}{2})Td(M){\rm ch}(E(W,q,\tau))\right.\\\notag£¬
&\left.\cdot{\rm det}^{\frac{1}{2}}\left(\frac{
{\rm sin}(\frac{1}{4 \pi^2}R^V)}{\frac{1}{4 \pi^2}R^V}\right)
{\rm ch}\left[
\bigotimes _{m=1}^{\infty}\wedge_{-q^m}(\widetilde{V_C})\right]
\right\}^{(2d)}.
\end{align}
\end{defn}
\begin{lem}We have
\begin{align}
&{\rm Ell}(M,W,V,\tau,z)=\left(\eta(\tau)^{3(d-la_0)}\prod_{i=1}^d\frac{2\pi \sqrt{-1}x_i}{\theta(\tau,x_i)}
\prod_{j=1}^l\prod_{r=1}^{a_0}\theta
(\tau,w_j-m_rz)\right.\\\notag
&\left.\cdot\prod_{r=1}^{b_0}\left(\frac{\theta_1(u_r,\tau)}{\theta_1(0,\tau)}\frac{\theta_2(u_r,\tau)}{\theta_2(0,\tau)}
\frac{\theta_3(u_r,\tau)}{\theta_3(0,\tau)}\right)
\right)^{(2d)},
\end{align}
\begin{align}
&\widetilde{{\rm Ell}}(M,W,V,\tau,z)=\left(\eta(\tau)^{3(d-la_0)}\prod_{i=1}^d\frac{2\pi \sqrt{-1}x_i}{\theta(\tau,x_i)}
\prod_{j=1}^l\prod_{r=1}^{a_0}\theta
(\tau,w_j-m_rz)\right.\\\notag
&\left.\cdot\prod_{r=1}^{b_0}\left(\frac{\theta_1(u_r,\tau)}{\theta_1(0,\tau)}+\frac{\theta_2(u_r,\tau)}{\theta_2(0,\tau)}
+\frac{\theta_3(u_r,\tau)}{\theta_3(0,\tau)}\right)
\right)^{(2d)},
\end{align}
\begin{align}
&\overline{{\rm Ell}}(M,W,V,\tau,z)=\left(\eta(\tau)^{3(d-la_0)}\prod_{i=1}^d\frac{2\pi \sqrt{-1}x_i}{\theta(\tau,x_i)}
\prod_{j=1}^l\prod_{r=1}^{a_0}\theta
(\tau,w_j-m_rz)\right.\\\notag
&\left.\cdot\prod_{r=1}^{b_0}\left(\frac{\theta(u_r,\tau)}{\theta_1(0,\tau)\theta_2(0,\tau)\theta_3(0,\tau)u_r}\right)
\right)^{(2d)}.
\end{align}
\end{lem}
One
has the following transformation laws of theta functions(cf. \cite{Ch} ):
\begin{equation}\theta_1(v,\tau+1)=e^{\frac{\pi\sqrt{-1}}{4}}\theta_1(v,\tau),~~\theta_1(v,-\frac{1}{\tau})
=\left(\frac{\tau}{\sqrt{-1}}\right)^{\frac{1}{2}}e^{\pi\sqrt{-1}\tau
v^2}\theta_2(\tau v,\tau);\end{equation}
\begin{equation}\theta_2(v,\tau+1)=\theta_3(v,\tau),~~\theta_2(v,-\frac{1}{\tau})
=\left(\frac{\tau}{\sqrt{-1}}\right)^{\frac{1}{2}}e^{\pi\sqrt{-1}\tau
v^2}\theta_1(\tau v,\tau);\end{equation}
\begin{equation}\theta_3(v,\tau+1)=\theta_2(v,\tau),~~\theta_3(v,-\frac{1}{\tau})
=\left(\frac{\tau}{\sqrt{-1}}\right)^{\frac{1}{2}}e^{\pi\sqrt{-1}\tau
v^2}\theta_3(\tau v,\tau).\end{equation}
By (2.19)-(2.21), similar to Theorem 2.3, we have

\begin{thm}
If $c_1(W)=0$ and the first Pontrjagin classes $p_1(M)=a_0p_1(W)$ and $p_1(E)=0$, then the generalized elliptic genus ${\rm Ell}(M,W,V,\tau,z)$,
$\widetilde{{\rm Ell}}(M,W,V,\tau,z)$,\\
 $\overline{{\rm Ell}}(M,W,V,\tau,z)$ satisfies (2.9) and (2.10).
\end{thm}

\section{Anomaly cancellation formulas for almost complex manifolds}
   We recall the Eisenstein series $G_{2k}(\tau)$ are defined to be
  \begin{equation}
  G_{2k}(\tau):=-\frac{B_{2k}}{4k}+\sum_{n=1}^{\infty}\sigma_{2k-1}(n)\cdot q^n,
\end{equation}
where $\sigma_k(n):=\sum_{m>0,~m|n}m^k$ and $B_{2k}$ are the Bernoulli numbers. It is well known that the whole grading ring of modular forms
over $SL(2,Z)$ are generated by $G_4(\tau)$ and $G_6(\tau)$. We recall Proposition 3.5 in \cite{LiP}
\begin{prop}(\cite{LiP}) Suppose a function $\varphi(\tau,z):\mathbb{H}\times \mathbb{C}\rightarrow \mathbb{C}$ satisfies
\begin{align}
\varphi(\frac{a\tau+b}{c\tau+d_0},\frac{z}{c\tau+d_0})=(c\tau+d_0)^{k}{\rm exp}(\frac{2\pi\sqrt{-1}mcz^2}{c\tau+d_0})\varphi(\tau,z);~~
\left(\begin{array}{cc}
\ a & b  \\
 c  & d_0
\end{array}\right)\in SL(2,Z).
\end{align}
We define
\begin{align}
&\Phi(\tau,z):={\rm exp}(-8\pi^2mG_2(\tau)z^2)\varphi(\tau,z):=\sum_{n\geq 0}a_n(\tau)\cdot z^n,
\end{align}
then these $a_n(\tau)$ are modular forms of weight $k+n$ over $SL(2,Z)$.
\end{prop}

\begin{prop}Let $c_1(W)=c_1(M)=0$ and the first Pontrjagin classes $p_1(M)=a_0p_1(W)$, then the series $a_n(M,W,\tau)$ determined by
\begin{align}
{\rm exp}(-4\pi^2l(\sum_{r=1}^{a_0}m_r)G_2(\tau)z^2){\rm Ell}(M,W,\tau,z)=\sum_{n\geq 0}a_n(M,W,\tau)\cdot z^n,
\end{align}
are modular forms of weight $d-la_0+n$ over $SL(2,Z)$. Furthermore, the first five series of $a_n(M,W,\tau)$ are of the following form:
\begin{align}
&a_0(M,W,\tau)=\left\{Td(M){\rm ch}(\wedge_{-1}(W_0^*))\right\}^{(2d)}+q\left\{Td(M){\rm ch}(\wedge_{-1}(W_0^*)){\rm ch}(A_0)\right\}^{(2d)}\\\notag
&
+q^2\left\{Td(M){\rm ch}(\wedge_{-1}(W_0^*)){\rm ch}(A_1)\right\}^{(2d)}+O(q^3),
\end{align}
where
\begin{align}
A_0=T+T^*-2(d-la_0)-W_0-W_0^*,~~W_0=a_0W,
\end{align}
and
\begin{align}
&A_1=S^2T+T^*\otimes T+S^2T^*+\wedge^2W_0^*+\wedge^2W_0+W_0^*\otimes W_0\\\notag
&+[2(d-la_0)-1](W_0+W_0^*-T-T^*)-(W_0+W_0^*)\otimes (T+T^*)\\\notag
&+
(d-la_0)(2d-2la_0-3),
\end{align}
\begin{align}
&a_1(M,W,\tau)=\left\{2\pi\sqrt{-1}Td(M){\rm ch}(\sum_{p_1,\cdots,p_{a_0}=0}^l(-1)^{\sum_{r=1}^{a_0}p_r}(\sum_{r=1}^{a_0}m_rp_r-\frac{l}{2}\sum_{r=1}^{a_0}m_r)\right.\\\notag
&\left.\wedge^{p_1}W^*\otimes\cdots\otimes\wedge^{p_{a_0}}W^*
)\right\}^{(2d)}+q\left\{Td(M){\rm ch}(A_3)\right\}^{(2d)}+O(q^2),
\end{align}
where
\begin{align}
&A_3=2\pi\sqrt{-1}[-2(d-la_0)+T+T^*-a_0(W+W^*)]\sum_{r=1}^{a_0}m_r(1+\wedge_{-1}(W^*))^{a_0-1}\otimes\\\notag
&(-W^*+2\wedge^2W^*+\cdots+(-1)^ll\wedge^lW^*)+(1+\wedge_{-1}(W^*))^{a_0}
\left\{-2\pi\sqrt{-1}\sum_{r=1}^{a_0}m_r(W+W^*)\right.\\\notag
&\left.-l\pi\sqrt{-1}\sum_{r=1}^{a_0}m_r[-2(d-la_0)+T+T^*-a_0(W+W^*)]\right\},
\end{align}
\begin{align}
&a_2(M,W,\tau)=\left\{-2\pi^2Td(M){\rm ch}(\sum_{p_1,\cdots,p_{a_0}=0}^l(-1)^{\sum_{r=1}^{a_0}p_r}(\sum_{r=1}^{a_0}m_rp_r-\frac{l}{2}\sum_{r=1}^{a_0}m_r)^2\right.\\\notag
&\left.\wedge^{p_1}W^*\otimes\cdots\otimes\wedge^{p_{a_0}}W^*
)\right\}^{(2d)}\\\notag
&+
\frac{l}{6}(\sum_{r=1}^{a_0}m_r^2)\pi^2\left\{Td(M){\rm ch}(\sum_{p_1,\cdots,p_{a_0}=0}^l(-1)^{\sum_{r=1}^{a_0}p_r}
\wedge^{p_1}W^*\otimes\cdots\otimes\wedge^{p_{a_0}}W^*
)\right\}^{(2d)}
+O(q),
\end{align}
\begin{align}
&a_3(M,W,\tau)=\left\{\frac{4}{3}\pi^3(\sqrt{-1})^3Td(M){\rm ch}(\sum_{p_1,\cdots,p_{a_0}=0}^l(-1)^{\sum_{r=1}^{a_0}p_r}(\sum_{r=1}^{a_0}m_rp_r-\frac{l}{2}\sum_{r=1}^{a_0}m_r)^3\right.\\\notag
&\left.\wedge^{p_1}W^*\otimes\cdots\otimes\wedge^{p_{a_0}}W^*
)\right\}^{(2d)}\\\notag
&+
\frac{\sqrt{-1}l}{3}(\sum_{r=1}^{a_0}m_r^2)\pi^3\left\{Td(M){\rm ch}(\sum_{p_1,\cdots,p_{a_0}=0}^l(-1)^{\sum_{r=1}^{a_0}p_r}
(\sum_{r=1}^{a_0}m_rp_r-\frac{l}{2}\sum_{r=1}^{a_0}m_r)\right.\\\notag
&\left.\wedge^{p_1}W^*\otimes\cdots\otimes\wedge^{p_{a_0}}W^*
)\right\}^{(2d)}
+O(q),
\end{align}
\begin{align}
&a_4(M,W,\tau)=\left\{\frac{2}{3}\pi^4Td(M){\rm ch}(\sum_{p_1,\cdots,p_{a_0}=0}^l(-1)^{\sum_{r=1}^{a_0}p_r}(\sum_{r=1}^{a_0}m_rp_r-\frac{l}{2}\sum_{r=1}^{a_0}m_r)^4\right.\\\notag
&\left.\wedge^{p_1}W^*\otimes\cdots\otimes\wedge^{p_{a_0}}W^*
)\right\}^{(2d)}\\\notag
&-
\frac{l^2}{3}(\sum_{r=1}^{a_0}m_r^2)^2\pi^4\left\{Td(M){\rm ch}(\sum_{p_1,\cdots,p_{a_0}=0}^l(-1)^{\sum_{r=1}^{a_0}p_r}
(\sum_{r=1}^{a_0}m_rp_r-\frac{l}{2}\sum_{r=1}^{a_0}m_r)^2\right.\\\notag
&\left.\wedge^{p_1}W^*\otimes\cdots\otimes\wedge^{p_{a_0}}W^*
)\right\}^{(2d)}\\\notag
&+\frac{l^2}{72}(\sum_{r=1}^{a_0}m_r^2)^2\pi^4\left\{Td(M)
{\rm ch}(
\sum_{p_1,\cdots,p_{a_0}=0}^l(-1)^{\sum_{r=1}^{a_0}p_r}
\wedge^{p_1}W^*\otimes\cdots\otimes\wedge^{p_{a_0}}W^*
)\right\}^{(2d)}
+O(q),
\end{align}
\end{prop}
\begin{proof} We know that ${\rm Ell}(M,W,\tau,0)=a_0(M,W,\tau)$ and
\begin{equation}{\rm Ell}(M,W,\tau,0)=\left\{Td(M){\rm ch}(E(M,W,\tau,0))\right\}^{(2d)}£¬
\end{equation}
where
\begin{align}
{\rm E}(M,W,\tau,0):=&c^{2(d-la_0)}\wedge_{-1}(W_0^*)\bigotimes _{n=1}^{\infty}
\wedge_{-q^n}(W^*_0)\wedge_{-q^{n}}(W_0)
\bigotimes\left(\bigotimes _{n=1}^{\infty}S_{q^n}(T^*)\otimes S_{q^n}(T)\right)\\\notag
&=\wedge_{-1}(W^*)+q\wedge_{-1}(W^*)\otimes A_0++q^2\wedge_{-1}(W^*)\otimes A_1+O(q^3).
\end{align}
So we get (3.5). If we set
\begin{equation}{\rm exp}(-4\pi^2l(\sum_{r=1}^{a_0}m_r)G_2(\tau)z^2):=C_0(Z)+C_1(z)q+O(q^2),
\end{equation}
and
\begin{equation}{\rm Ell}(M,W,\tau,z):=B_0(Z)+B_1(z)q+O(q^2),
\end{equation}
we can get that
\begin{align}
&C_0(z)=1+\frac{l}{6}(\sum_{r=1}^{a_0}m_r^2)\pi^2z^2+\frac{l^2}{72}(\sum_{r=1}^{a_0}m_r^2)^2\pi^4z^4+O(z^6),\\\notag
&C_1(z)=-4l(\sum_{r=1}^{a_0}m_r^2)\pi^2z^2-\frac{2l^2}{3}(\sum_{r=1}^{a_0}m_r^2)^2\pi^4z^4+O(z^6),\\\notag
&B_0(z)=\left\{Td(M){\rm ch}(\wedge_{-1}(W_0^*))\right\}^{(2d)}\\\notag
&+
\left\{2\pi\sqrt{-1}Td(M){\rm ch}(\sum_{p_1,\cdots,p_{a_0}=0}^l(-1)^{\sum_{r=1}^{a_0}p_r}(\sum_{r=1}^{a_0}m_rp_r-\frac{l}{2}\sum_{r=1}^{a_0}m_r)\right.\\\notag
&\left.\wedge^{p_1}W^*\otimes\cdots\otimes\wedge^{p_{a_0}}W^*
)\right\}^{(2d)}z\\\notag
&+\left\{-2\pi^2Td(M){\rm ch}(\sum_{p_1,\cdots,p_{a_0}=0}^l(-1)^{\sum_{r=1}^{a_0}p_r}(\sum_{r=1}^{a_0}m_rp_r-\frac{l}{2}\sum_{r=1}^{a_0}m_r)^2\right.\\\notag
&\left.\wedge^{p_1}W^*\otimes\cdots\otimes\wedge^{p_{a_0}}W^*
)\right\}^{(2d)}z^2\\\notag
&+
\left\{\frac{4}{3}\pi^3(\sqrt{-1})^3Td(M){\rm ch}(\sum_{p_1,\cdots,p_{a_0}=0}^l(-1)^{\sum_{r=1}^{a_0}p_r}(\sum_{r=1}^{a_0}m_rp_r-\frac{l}{2}\sum_{r=1}^{a_0}m_r)^3\right.\\\notag
&\left.\wedge^{p_1}W^*\otimes\cdots\otimes\wedge^{p_{a_0}}W^*
)\right\}^{(2d)}z^3\\\notag
&+\left\{\frac{2}{3}\pi^4Td(M){\rm ch}(\sum_{p_1,\cdots,p_{a_0}=0}^l(-1)^{\sum_{r=1}^{a_0}p_r}(\sum_{r=1}^{a_0}m_rp_r-\frac{l}{2}\sum_{r=1}^{a_0}m_r)^4\right.\\\notag
&\left.\wedge^{p_1}W^*\otimes\cdots\otimes\wedge^{p_{a_0}}W^*
)\right\}^{(2d)}z^4+O(z^5)\\\notag
&B_1(y)=
\left\{Td(M){\rm ch}(\wedge_{-1}(W^*_0)\otimes[-2(d-la_0)+T+T^*-a_0(W+W^*)])\right\}^{(2d)}\\\notag
&+\left\{Td(M){\rm ch}(2\pi\sqrt{-1}[-2(d-la_0)+T+T^*-a_0(W+W^*)]\sum_{r=1}^{a_0}m_r(1+\wedge_{-1}(W^*))^{a_0-1}\right.\otimes\\\notag
&(-W^*+2\wedge^2W^*+\cdots+(-1)^ll\wedge^lW^*)+(1+\wedge_{-1}(W^*))^{a_0}
\left\{-2\pi\sqrt{-1}\sum_{r=1}^{a_0}m_r(W+W^*)\right.\\\notag
&\left.-l\pi\sqrt{-1}\sum_{r=1}^{a_0}m_r[-2(d-la_0)+T+T^*-a_0(W+W^*)]\right\}^{(2d)}z+O(z^2).
\end{align}
We know that
\begin{equation}
\sum_{n\geq 0}a_n(M,W,\tau)\cdot z^n=C_0(z)B_0(z)+[C_0(z)B_1(z)+C_1(z)B_0(z)]q+\cdots,
\end{equation}
then we can get Proposition 3.2 by (3.17) and (3.18).
\end{proof}
Since there are no $SL(2,Z)$ modular forms with the odd weight or the non zero weight $\leq 2$, we have
\begin{prop}
Let $c_1(W)=c_1(M)=0$ and the first Pontrjagin classes $p_1(M)=a_0p_1(W)$, then\\
1)if either $d-la_0$ is odd or $d-la_0\leq 2$ but $d-la_0\neq 0$, then
\begin{align}
&\left\{Td(M){\rm ch}(\wedge_{-1}(W_0^*))\right\}^{(2d)}=\left\{Td(M){\rm ch}(\wedge_{-1}(W_0^*)){\rm ch}(A_0)\right\}^{(2d)}=0\\\notag
&\left\{Td(M){\rm ch}(\wedge_{-1}(W_0^*)){\rm ch}(A_1)\right\}^{(2d)}=0,
\end{align}
and for complex manifolds,
\begin{align}
&\chi(M,\wedge_{-1}(W_0^*))=\chi(M,\wedge_{-1}(W_0^*)\otimes (A_0))=\chi(M,\wedge_{-1}(W_0^*)\otimes (A_1))=0,
\end{align}
where $\chi(M,\wedge_{-1}(W_0^*))$ denotes the twisted holomorphic Euler characteristic number. \\
2)if either $d-la_0$ is even or $d-la_0\leq 1$ but $d-la_0\neq -1$, then
\begin{align}
&\left\{Td(M){\rm ch}(\sum_{p_1,\cdots,p_{a_0}=0}^l(-1)^{\sum_{r=1}^{a_0}p_r}(\sum_{r=1}^{a_0}m_rp_r-\frac{l}{2}\sum_{r=1}^{a_0}m_r)\right.\\\notag
&\left.\wedge^{p_1}W^*\otimes\cdots\otimes\wedge^{p_{a_0}}W^*
)\right\}^{(2d)}=0,\\\notag
&\left\{Td(M){\rm ch}(A_3)\right\}^{(2d)}=0,
\end{align}
and for complex manifolds,
\begin{align}
&\chi(M,\sum_{p_1,\cdots,p_{a_0}=0}^l(-1)^{\sum_{r=1}^{a_0}p_r}(\sum_{r=1}^{a_0}m_rp_r-\frac{l}{2}\sum_{r=1}^{a_0}m_r)
\wedge^{p_1}W^*\otimes\cdots\otimes\wedge^{p_{a_0}}W^*)=0,\\\notag
&\chi(M,A_3)=0.
\end{align}
3)if either $d-la_0$ is odd or $d-la_0\leq 0$ but $d-la_0\neq -2$, then
\begin{align}
&\left\{-2Td(M){\rm ch}(\sum_{p_1,\cdots,p_{a_0}=0}^l(-1)^{\sum_{r=1}^{a_0}p_r}(\sum_{r=1}^{a_0}m_rp_r-\frac{l}{2}\sum_{r=1}^{a_0}m_r)^2\right.\\\notag
&\left.\wedge^{p_1}W^*\otimes\cdots\otimes\wedge^{p_{a_0}}W^*
)\right\}^{(2d)}\\\notag
&+
\frac{l}{6}(\sum_{r=1}^{a_0}m_r^2)\left\{Td(M){\rm ch}(\sum_{p_1,\cdots,p_{a_0}=0}^l(-1)^{\sum_{r=1}^{a_0}p_r}
\wedge^{p_1}W^*\otimes\cdots\otimes\wedge^{p_{a_0}}W^*
)\right\}^{(2d)}=0,
\end{align}
4)if either $d-la_0$ is even or $d-la_0\leq -1$ but $d-la_0\neq -3$, then
\begin{align}
&\left\{\frac{4}{3}\pi^3(\sqrt{-1})^3Td(M){\rm ch}(\sum_{p_1,\cdots,p_{a_0}=0}^l(-1)^{\sum_{r=1}^{a_0}p_r}(\sum_{r=1}^{a_0}m_rp_r-\frac{l}{2}\sum_{r=1}^{a_0}m_r)^3\right.\\\notag
&\left.\wedge^{p_1}W^*\otimes\cdots\otimes\wedge^{p_{a_0}}W^*
)\right\}^{(2d)}\\\notag
&+
\frac{\sqrt{-1}l}{3}(\sum_{r=1}^{a_0}m_r^2)\pi^3\left\{Td(M){\rm ch}(\sum_{p_1,\cdots,p_{a_0}=0}^l(-1)^{\sum_{r=1}^{a_0}p_r}
(\sum_{r=1}^{a_0}m_rp_r-\frac{l}{2}\sum_{r=1}^{a_0}m_r)\right.\\\notag
&\left.\wedge^{p_1}W^*\otimes\cdots\otimes\wedge^{p_{a_0}}W^*
)\right\}^{(2d)}=0,
\end{align}
5)if either $d-la_0$ is odd or $d-la_0\leq -2$ but $d-la_0\neq -4$, then
\begin{align}
&\left\{\frac{2}{3}\pi^4Td(M){\rm ch}(\sum_{p_1,\cdots,p_{a_0}=0}^l(-1)^{\sum_{r=1}^{a_0}p_r}(\sum_{r=1}^{a_0}m_rp_r-\frac{l}{2}\sum_{r=1}^{a_0}m_r)^4\right.\\\notag
&\left.\wedge^{p_1}W^*\otimes\cdots\otimes\wedge^{p_{a_0}}W^*
)\right\}^{(2d)}\\\notag
&-
\frac{l^2}{3}(\sum_{r=1}^{a_0}m_r^2)^2\pi^4\left\{Td(M){\rm ch}(\sum_{p_1,\cdots,p_{a_0}=0}^l(-1)^{\sum_{r=1}^{a_0}p_r}
(\sum_{r=1}^{a_0}m_rp_r-\frac{l}{2}\sum_{r=1}^{a_0}m_r)^2\right.\\\notag
&\left.\wedge^{p_1}W^*\otimes\cdots\otimes\wedge^{p_{a_0}}W^*
)\right\}^{(2d)}\\\notag
&+\frac{l^2}{72}(\sum_{r=1}^{a_0}m_r^2)^2\pi^4\left\{Td(M)
{\rm ch}(
\sum_{p_1,\cdots,p_{a_0}=0}^l(-1)^{\sum_{r=1}^{a_0}p_r}
\wedge^{p_1}W^*\otimes\cdots\otimes\wedge^{p_{a_0}}W^*
)\right\}^{(2d)}=0.
\end{align}
Similarly in cases 3)4)5), we have expressions using the twisted holomorphic Euler characteristic number.
\end{prop}

\begin{thm}Let $c_1(W)=c_1(M)=0$ and the first Pontrjagin classes $p_1(M)=a_0p_1(W)$, then\\
1)if $d-la_0=4$, then
\begin{align}
&\left\{Td(M){\rm ch}(\wedge_{-1}(W_0^*)){\rm ch}(A_0)\right\}^{(2d)}=240\left\{Td(M){\rm ch}(\wedge_{-1}(W_0^*))\right\}^{(2d)}\\\notag
&\left\{Td(M){\rm ch}(\wedge_{-1}(W_0^*)){\rm ch}(A_1)\right\}^{(2d)}=2160\left\{Td(M){\rm ch}(\wedge_{-1}(W_0^*))\right\}^{(2d)},
\end{align}
and for complex manifolds,
\begin{align}
&\chi(M,\wedge_{-1}(W_0^*)\otimes A_0)=240\chi(M,\wedge_{-1}(W_0^*)),\\\notag
&\chi(M,\wedge_{-1}(W_0^*)\otimes A_1)=2160\chi(M,\wedge_{-1}(W_0^*)),
\end{align}
so $\chi(M,\wedge_{-1}(W_0^*)\otimes A_0)$ is the integer multiple of $240$
and $\chi(M,\wedge_{-1}(W_0^*)\otimes A_1)$ is the integer multiple of $2160$.

2)if $d-la_0=6$, then
\begin{align}
&\left\{Td(M){\rm ch}(\wedge_{-1}(W_0^*)){\rm ch}(A_0)\right\}^{(2d)}=-504\left\{Td(M){\rm ch}(\wedge_{-1}(W_0^*))\right\}^{(2d)}\\\notag
&\left\{Td(M){\rm ch}(\wedge_{-1}(W_0^*)){\rm ch}(A_1)\right\}^{(2d)}=-16632\left\{Td(M){\rm ch}(\wedge_{-1}(W_0^*))\right\}^{(2d)},
\end{align}
and for complex manifolds,
\begin{align}
&\chi(M,\wedge_{-1}(W_0^*)\otimes A_0)=-504\chi(M,\wedge_{-1}(W_0^*)),\\\notag
&\chi(M,\wedge_{-1}(W_0^*)\otimes A_1)=-16632\chi(M,\wedge_{-1}(W_0^*)),
\end{align}
so $\chi(M,\wedge_{-1}(W_0^*)\otimes A_0)$ is the integer multiple of $504$
and $\chi(M,\wedge_{-1}(W_0^*)\otimes A_1)$ is the integer multiple of $16632$.

3)if $d-la_0=8$, then
\begin{align}
&\left\{Td(M){\rm ch}(\wedge_{-1}(W_0^*)){\rm ch}(A_0)\right\}^{(2d)}=480\left\{Td(M){\rm ch}(\wedge_{-1}(W_0^*))\right\}^{(2d)}\\\notag
&\left\{Td(M){\rm ch}(\wedge_{-1}(W_0^*)){\rm ch}(A_1)\right\}^{(2d)}=61920\left\{Td(M){\rm ch}(\wedge_{-1}(W_0^*))\right\}^{(2d)},
\end{align}
and for complex manifolds,
\begin{align}
&\chi(M,\wedge_{-1}(W_0^*)\otimes A_0)=480\chi(M,\wedge_{-1}(W_0^*)),\\\notag
&\chi(M,\wedge_{-1}(W_0^*)\otimes A_1)=61920\chi(M,\wedge_{-1}(W_0^*)),
\end{align}
so $\chi(M,\wedge_{-1}(W_0^*)\otimes A_0)$ is the integer multiple of $480$
and $\chi(M,\wedge_{-1}(W_0^*)\otimes A_1)$ is the integer multiple of $61920$.

4)if $d-la_0=10$, then
\begin{align}
&\left\{Td(M){\rm ch}(\wedge_{-1}(W_0^*)){\rm ch}(A_0)\right\}^{(2d)}=-264\left\{Td(M){\rm ch}(\wedge_{-1}(W_0^*))\right\}^{(2d)}\\\notag
&\left\{Td(M){\rm ch}(\wedge_{-1}(W_0^*)){\rm ch}(A_1)\right\}^{(2d)}=-135432\left\{Td(M){\rm ch}(\wedge_{-1}(W_0^*))\right\}^{(2d)},
\end{align}
and for complex manifolds,
\begin{align}
&\chi(M,\wedge_{-1}(W_0^*)\otimes A_0)=-264\chi(M,\wedge_{-1}(W_0^*)),\\\notag
&\chi(M,\wedge_{-1}(W_0^*)\otimes A_1)=-135432\chi(M,\wedge_{-1}(W_0^*)),
\end{align}
so $\chi(M,\wedge_{-1}(W_0^*)\otimes A_0)$ is the integer multiple of $264$
and $\chi(M,\wedge_{-1}(W_0^*)\otimes A_1)$ is the integer multiple of $135432$.

5)if $d-la_0=12$, then
\begin{align}
&\left\{Td(M){\rm ch}(\wedge_{-1}(W_0^*)){\rm ch}(A_1)\right\}^{(2d)}=196560
\left\{Td(M){\rm ch}(\wedge_{-1}(W_0^*))\right\}^{(2d)}\\\notag
&-24\left\{Td(M){\rm ch}(\wedge_{-1}(W_0^*)){\rm ch}(A_0)\right\}^{(2d)},
\end{align}
and for complex manifolds,
$\chi(M,\wedge_{-1}(W_0^*)\otimes A_1)$ is the integer multiple of $24$.

6)if $d-la_0=14$, then
\begin{align}
&\left\{Td(M){\rm ch}(\wedge_{-1}(W_0^*)){\rm ch}(A_0)\right\}^{(2d)}=-24\left\{Td(M){\rm ch}(\wedge_{-1}(W_0^*))\right\}^{(2d)}\\\notag
&\left\{Td(M){\rm ch}(\wedge_{-1}(W_0^*)){\rm ch}(A_1)\right\}^{(2d)}=-196632\left\{Td(M){\rm ch}(\wedge_{-1}(W_0^*))\right\}^{(2d)},
\end{align}
and for complex manifolds,
\begin{align}
&\chi(M,\wedge_{-1}(W_0^*)\otimes A_0)=-24\chi(M,\wedge_{-1}(W_0^*)),\\\notag
&\chi(M,\wedge_{-1}(W_0^*)\otimes A_1)=-196632\chi(M,\wedge_{-1}(W_0^*)),
\end{align}
so $\chi(M,\wedge_{-1}(W_0^*)\otimes A_0)$ is the integer multiple of $24$
and $\chi(M,\wedge_{-1}(W_0^*)\otimes A_1)$ is the integer multiple of $196632$.

7)if $d-la_0=16$, then
\begin{align}
&\left\{Td(M){\rm ch}(\wedge_{-1}(W_0^*)){\rm ch}(A_1)\right\}^{(2d)}=146880
\left\{Td(M){\rm ch}(\wedge_{-1}(W_0^*))\right\}^{(2d)}\\\notag
&+216\left\{Td(M){\rm ch}(\wedge_{-1}(W_0^*)){\rm ch}(A_0)\right\}^{(2d)},
\end{align}
and for complex manifolds,
$\chi(M,\wedge_{-1}(W_0^*)\otimes A_1)$ is the integer multiple of $216$.
\end{thm}

\begin{proof}
$a_0(M,W,\tau)$ is a modular form of weight $d-la_0$ over $SL(2,Z)$. Consequently \\
1)if $d-la_0=4$, then $a_0(M,W,\tau)$ is proportional to
$$
G_4(\tau)=1+240q+2160q^2+6720q^3+\cdots,
$$
so (3.26) holds.\\
2)if $d-la_0=6$, then $a_0(M,W,\tau)$ is proportional to
$$
G_6(\tau)=1-504q-16632q^2-122976q^3+\cdots.
$$
so (3.28) holds.\\
3)if $d-la_0=8$, then $a_0(M,W,\tau)$ is proportional to
$$
G_4(\tau)^2=1+480q+61920q^2+\cdots.
$$
so (3.30) holds.\\
4)if $d-la_0=10$, then $a_0(M,W,\tau)$ is proportional to
$$
G_4(\tau)G_6(\tau)=1-264q-135432q^2+\cdots.
$$
so (3.32) holds.\\
5)if $d-la_0=12$, then
\begin{equation}a_0(M,W,\tau)=\lambda_1 G_4(\tau)^3+\lambda_2 G_6(\tau)^2,
\end{equation}
where $\lambda_1,\lambda_2$ are degree $2d$ forms. We have
\begin{equation}
G_4(\tau)^3=1+720q+179280q^2+\cdots,
\end{equation}
\begin{equation}
G_6(\tau)^2=1-1008q+220752q^2+\cdots.
\end{equation}
In (3.38), we compare the coefficients of $1$, $q$, $q^2$, we get three equations about $\lambda_1$, $\lambda_2$. Solve the three equations,
we get (3.34).\\
6)if $d-la_0=14$, then $a_0(M,W,\tau)$ is proportional to
$$
G_4(\tau)^2G_6(\tau)=1-24q-196632q^2+\cdots,
$$
so (3.35) holds.\\
7)if $d-la_0=16$, then
\begin{equation}a_0(M,W,\tau)=\lambda_1 G_4(\tau)^4+\lambda_2 G_4(\tau)G_6(\tau)^2,
\end{equation}
where $\lambda_1,\lambda_2$ are degree $2d$ forms. We have
\begin{equation}
G_4(\tau)^4=1+960q+354240q^2+\cdots,
\end{equation}
\begin{equation}
G_4(\tau)G_6(\tau)^2=1-768q-19008q^2+\cdots.
\end{equation}
By (3.41)-(3.43), we get (3.37).
\end{proof}

Similarly, by $a_1(M,W,\tau)$, we have
\begin{thm}Let $c_1(W)=c_1(M)=0$ and the first Pontrjagin classes $p_1(M)=a_0p_1(W)$, then\\
1)if $d-la_0=3$, then
\begin{align}
&\left\{Td(M){\rm ch}(\wedge_{-1}(W_0^*)){\rm ch}(A_3)\right\}^{(2d)}\\
&=240
\left\{2\pi\sqrt{-1}Td(M){\rm ch}(\sum_{p_1,\cdots,p_{a_0}=0}^l(-1)^{\sum_{r=1}^{a_0}p_r}(\sum_{r=1}^{a_0}m_rp_r-\frac{l}{2}\sum_{r=1}^{a_0}m_r)\right.\\\notag
&\left.\wedge^{p_1}W^*\otimes\cdots\otimes\wedge^{p_{a_0}}W^*
)\right\}^{(2d)}
\end{align}
and for complex manifolds, $\chi(M,\wedge_{-1}(W_0^*)\otimes A_3)$ is the integer multiple of $240$.\\
2)if $d-la_0=5$, then
\begin{align}
&\left\{Td(M){\rm ch}(\wedge_{-1}(W_0^*)){\rm ch}(A_3)\right\}^{(2d)}\\
&=-504
\left\{2\pi\sqrt{-1}Td(M){\rm ch}(\sum_{p_1,\cdots,p_{a_0}=0}^l(-1)^{\sum_{r=1}^{a_0}p_r}(\sum_{r=1}^{a_0}m_rp_r-\frac{l}{2}\sum_{r=1}^{a_0}m_r)\right.\\\notag
&\left.\wedge^{p_1}W^*\otimes\cdots\otimes\wedge^{p_{a_0}}W^*
)\right\}^{(2d)}
\end{align}
and for complex manifolds, $\chi(M,\wedge_{-1}(W_0^*)\otimes A_3)$ is the integer multiple of $504$.\\
3)if $d-la_0=7$, then
\begin{align}
&\left\{Td(M){\rm ch}(\wedge_{-1}(W_0^*)){\rm ch}(A_3)\right\}^{(2d)}\\
&=480
\left\{2\pi\sqrt{-1}Td(M){\rm ch}(\sum_{p_1,\cdots,p_{a_0}=0}^l(-1)^{\sum_{r=1}^{a_0}p_r}(\sum_{r=1}^{a_0}m_rp_r-\frac{l}{2}\sum_{r=1}^{a_0}m_r)\right.\\\notag
&\left.\wedge^{p_1}W^*\otimes\cdots\otimes\wedge^{p_{a_0}}W^*
)\right\}^{(2d)}
\end{align}
and for complex manifolds, $\chi(M,\wedge_{-1}(W_0^*)\otimes A_3)$ is the integer multiple of $480$.\\
4)if $d-la_0=9$, then
\begin{align}
&\left\{Td(M){\rm ch}(\wedge_{-1}(W_0^*)){\rm ch}(A_3)\right\}^{(2d)}\\
&=-264
\left\{2\pi\sqrt{-1}Td(M){\rm ch}(\sum_{p_1,\cdots,p_{a_0}=0}^l(-1)^{\sum_{r=1}^{a_0}p_r}(\sum_{r=1}^{a_0}m_rp_r-\frac{l}{2}\sum_{r=1}^{a_0}m_r)\right.\\\notag
&\left.\wedge^{p_1}W^*\otimes\cdots\otimes\wedge^{p_{a_0}}W^*
)\right\}^{(2d)}
\end{align}
and for complex manifolds, $\chi(M,\wedge_{-1}(W_0^*)\otimes A_3)$ is the integer multiple of $264$.
\end{thm}
We remark that the idea of proving Theorem 3.4 and Theorem 3.5 appears in \cite{LiP}.
\section{Acknowledgements}

 The author was supported in part by NSFC No.11771070. The author is indebted to Prof. P. Li for helpful comments. The author also thank the referee for his (or her) careful reading and helpful comments.

\vskip 1 true cm


\bigskip
\bigskip

 \indent{School of Mathematics and Statistics,
Northeast Normal University, Changchun Jilin, 130024, China }\\
\indent E-mail: {\it wangy581@nenu.edu.cn }\\

\end{document}